\documentclass[12pt]{article}
\usepackage{amsthm,geometry,amssymb,amsmath,enumerate,float,tikz,cite,algorithm2e,setspace}
\newtheorem{theorem}{Theorem}[section]
\newtheorem{proposition}{Proposition}[section]
\newtheorem{lemma}{Lemma}[section]

\usepackage{lineno}

\title{Algorithmic aspects of broadcast independence}
\author{S. Bessy$^1$ \and  D. Rautenbach$^2$}
\date{}

\begin{document}
\onehalfspace
%\linenumbers

\maketitle
\vspace{-10mm}
\begin{center}
{\small $^1$ Laboratoire d'Informatique, de Robotique et de
  Micro\'{e}lectronique de Montpellier,\\ Montpellier, France,
  \texttt{stephane.bessy@lirmm.fr}\\[3mm] $^2$ Institute of
  Optimization and Operations Research, Ulm University,\\ Ulm,
  Germany, \texttt{dieter.rautenbach@uni-ulm.de}}
\end{center}

\begin{abstract}
An independent broadcast on a connected graph $G$ is a function
$f:V(G)\to \mathbb{N}_0$ such that, for every vertex $x$ of $G$, the
value $f(x)$ is at most the eccentricity of $x$ in $G$, and $f(x)>0$
implies that $f(y)=0$ for every vertex $y$ of $G$ within distance at
most $f(x)$ from $x$.  The broadcast independence number $\alpha_b(G)$
of $G$ is the largest weight $\sum\limits_{x\in V(G)}f(x)$ of an
independent broadcast $f$ on $G$.

We describe an efficient algorithm that determines the broadcast
independence number of a given tree.  Furthermore, we show NP-hardness
of the broadcast independence number for planar graphs of maximum
degree four, and hardness of approximation for general graphs.  Our
results solve problems posed by Dunbar, Erwin, Haynes, Hedetniemi, and
Hedetniemi (2006), Hedetniemi (2006), and Ahmane, Bouchemakh, Sopena
(2018).
\end{abstract}
{\small 
\begin{tabular}{lp{13cm}}
{\bf Keywords}: broadcast independence\\
\end{tabular}
}

\pagebreak

\section{Introduction}

In his PhD thesis \cite{er} Erwin introduced the notions of broadcast
domination and broadcast independence in graphs.  While broadcast
domination was studied in detail, only little research has been done
on broadcast independence~\cite{ahboso,boze,duerhahehe,he}, and several fundamental
problems related to this notion remained open.  After efficient
algorithms for optimal broadcast domination were developed for
restricted graph classes \cite{blhehoma,dadehe}, Heggernes and
Lokshtanov \cite{helo} showed the beautiful and surprising result that
broadcast domination can be solved optimally in polynomial time for
every graph.  In contrast to that, Dunbar et al.~\cite{duerhahehe} and
Hedetniemi \cite{he} explicitly ask about the complexity of broadcast
independence and about efficient algorithms for trees.  As pointed out
recently by Ahmane et al.~\cite{ahboso}, the complexity of broadcast
independence was unknown even for trees.

In the present paper we describe an efficient algorithm for optimal
broadcast independence in trees.  Furthermore, we show hardness of
approximation for general graphs and NP-completeness for planar graphs
of maximum degree four.  
Before stating our results precisely, 
we collect the necessary definitions.  We
consider finite, simple, and undirected graphs, and use standard
terminology and notation.  Let $\mathbb{N}_0$ be the set of
nonnegative integers, and let $\mathbb{Z}$ be the set of integers.
For a connected graph $G$, a function $f:V(G)\to \mathbb{N}_0$ is an
{\it independent broadcast on $G$} if
\begin{quote}
\begin{enumerate}[(B1)]
\item $f(x)\leq {\rm ecc}_G(x)$ for every vertex $x$ of $G$, where
  ${\rm ecc}_G(x)$ is the eccentricity of $x$ in $G$, and
\item ${\rm dist}_G(x,y)>\max\{ f(x),f(y)\}$ for every two distinct
  vertices $x$ and $y$ of $G$ with $f(x),f(y)>0$, where ${\rm
    dist}_G(x,y)$ is the distance of $x$ and $y$ in $G$.
\end{enumerate}
\end{quote}
The {\it weight} of $f$ is $\sum\limits_{x\in V(G)}f(x)$.  
The {\it broadcast independence number} $\alpha_b(G)$ 
of $G$ is the maximum
weight of an independent broadcast on $G$, and an independent
broadcast on $G$ of weight $\alpha_b(G)$ is {\it optimal}.  Let
$\alpha(G)$ be the usual independence number of $G$
defined as the maximum cardinality of an independent set in $G$,
which is a set of pairwise nonadjacent vertices of $G$.
For an integer
$k$, let $[k]$ be the set of all positive integers at most $k$, and
let $[k]_0=\{ 0\}\cup [k]$.

Note that adding a universal vertex to a non-empty graph does not
change its independence number but reduces its diameter to two, and
that $\alpha_b(G)=\alpha(G)$ for a graph $G$ with diameter two and
$\alpha(G)\geq 3$.  These observations imply that Zuckerman's
\cite{zu} hardness of approximation result for maximum clique
immediately yields the following.

\begin{proposition}\label{proposition1}
For every positive real number $\epsilon$, it is NP-hard to
approximate the broadcast independence number 
of a given connected graph of
order $n$ to within $n^{1-\epsilon}$.
\end{proposition}
As already stated, we show that computing the broadcast independence
number remains hard even when restricted to instances with bounded
maximum degree.
In fact, we believe that it is hard even when restricted to cubic graphs.
\begin{theorem}\label{theorem1}
For a given connected planar graph $G$ of maximum degree $4$ 
and a given positive integer $k$, 
it is NP-complete to decide whether
$\alpha_b(G)\geq k$.
\end{theorem}
Clearly, $\alpha_b(G)\geq \alpha(G)$ for every connected graph $G$.
In \cite{bera} we show
$\alpha_b(G)\leq 4\alpha(G)$ for every connected graph $G$,
which yields efficient constant factor approximation algorithms 
for the broadcast independence number
on every class of connected graphs 
on which the independence number 
can efficiently be approximated within a constant factor;
in particular, on graphs of bounded maximum degree.

In the next section, we prove Theorem~\ref{theorem1}, and, 
in Section~\ref{section:polyAlgoTree}, 
we present the polynomial time algorithm to
compute the broadcast independence number of a given tree, 
more precisely, we prove the following.

\begin{theorem}\label{theorem2}
The broadcast independence number $\alpha_b(T)$ of a given tree $T$ of
order $n$ can be determined in $O(n^9)$ time.
\end{theorem}

\section{NP-completeness of broadcast independence on planar graphs
with maximum degree $5$}
  \label{section:hardnessMaxDegree5}

In this section, we prove Theorem~\ref{theorem1}.

Since an independent broadcast can be encoded using polynomially many
bits, and (B1) and (B2) can be checked in polynomial time, the
considered decision problem is in NP.
We show its NP-completeness by reducing to it the NP-complete problem
\cite{gajo} {\sc Independent Set} restricted to connected planar cubic
graphs.  
Therefore, let $(H,k)$ be an instance of {\sc Independent Set}, where $H$ is a connected planar cubic graph.
Recall that {\sc Independent Set} is the problem to decide whether
$\alpha(H)\geq k$.  In order to complete the proof, we describe a
polynomial time construction of a planar graph $G$ of maximum degree
$4$ such that 
$$\alpha_b(G)=\alpha(H)+\frac{45}{2}n(H),$$ 
where $n(H)$ is the order of $H$.  
First, let the graph $H'$ arise 
by subdividing each edge of $H$ exactly twice.
It is well known and easy to see that 
$\alpha(H')=\alpha(H)+m(H)=\alpha(H)+\frac{3}{2}n(H)$,
where $m(H)$ is the number of edges of $H$.
Now, the graph $G$ arises from $H'$ by
\begin{itemize}
\item adding, for every vertex $x$ in $V(H)$, 
one copy $K(x)$ of the star
$K_{1,4}$ of order $4$ and connecting its center with $x$, and
\item adding, for every vertex $x$ in $V(H')\setminus V(H)$, 
two disjoint copies $K_1(x)$ and $K_2(x)$ of the star
$K_{1,4}$ of order $4$ and connecting their two centers with $x$.
\end{itemize}
See Figure \ref{fig1} for an illustration.

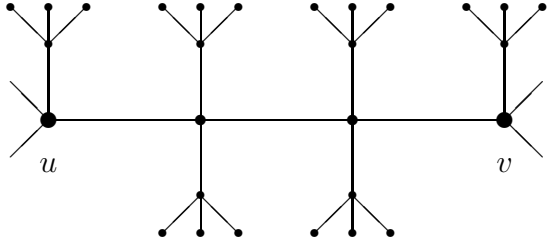
\begin{figure}[H]
\begin{center}
%TeXCAD Picture [5.pic]. Options:
%\grade{\on}
%\emlines{\off}
%\epic{\off}
%\beziermacro{\on}
%\reduce{\on}
%\snapping{\on}
%\pvinsert{% Your \input, \def, etc. here}
%\quality{8.000}
%\graddiff{0.005}
%\snapasp{1}
%\zoom{11.3137}
\unitlength 1mm % = 2.845pt
\linethickness{0.4pt}
\ifx\plotpoint\undefined\newsavebox{\plotpoint}\fi % GNUPLOT compatibility
\begin{picture}(70.5,31)(0,0)
\put(5,15){\circle*{2}}
\put(65,15){\circle*{2}}
\put(5,15){\line(1,0){60}}
\put(25,15){\circle*{1.5}}
\put(45,15){\circle*{1.5}}
\put(25,5){\circle*{1}}
\put(45,5){\circle*{1}}
\put(20,0){\circle*{1}}
\put(40,0){\circle*{1}}
\put(30,0){\circle*{1}}
\put(50,0){\circle*{1}}
\put(25,25){\circle*{1}}
\put(45,25){\circle*{1}}
\put(65,25){\circle*{1}}
\put(5,25){\circle*{1}}
\put(30,30){\circle*{1}}
\put(50,30){\circle*{1}}
\put(70,30){\circle*{1}}
\put(10,30){\circle*{1}}
\put(20,30){\circle*{1}}
\put(40,30){\circle*{1}}
\put(60,30){\circle*{1}}
\put(0,30){\circle*{1}}
\put(20,30){\line(1,-1){5}}
\put(40,30){\line(1,-1){5}}
\put(60,30){\line(1,-1){5}}
\put(0,30){\line(1,-1){5}}
\put(25,25){\line(1,1){5}}
\put(45,25){\line(1,1){5}}
\put(65,25){\line(1,1){5}}
\put(5,25){\line(1,1){5}}
\put(25,25){\line(0,-1){20}}
\put(45,25){\line(0,-1){20}}
\put(30,0){\line(-1,1){5}}
\put(50,0){\line(-1,1){5}}
\put(25,5){\line(-1,-1){5}}
\put(45,5){\line(-1,-1){5}}
\put(5,25){\line(0,-1){10}}
\put(65,25){\line(0,-1){10}}
\put(5,9){\makebox(0,0)[cc]{$u$}}
\put(65,9){\makebox(0,0)[cc]{$v$}}
\put(70,20){\line(-1,-1){5}}
\put(65,15){\line(1,-1){5}}
\put(0,20){\line(1,-1){5}}
\put(5,15){\line(-1,-1){5}}
\put(5,30){\circle*{1}}
\put(25,30){\circle*{1}}
\put(45,30){\circle*{1}}
\put(65,30){\circle*{1}}
\put(5,30){\line(0,-1){5}}
\put(25,30){\line(0,-1){5}}
\put(45,30){\line(0,-1){5}}
\put(65,30){\line(0,-1){5}}
\put(25,0){\circle*{1}}
\put(45,0){\circle*{1}}
\put(25,5){\line(0,-1){5}}
\put(45,5){\line(0,-1){5}}
\end{picture}
\end{center}
\caption{An edge $uv$ of $H$ after two subdivisions and the attachment of the disjoint stars. 
The vertices of $H$ are shown largest and
the vertices of the attached stars are shown smallest.}\label{fig1}
\end{figure}
Note that $G$ is connected and planar, 
has order $32n(H)$ and maximum degree $4$, 
and contains $21n(H)$ endvertices.  
Since some maximum independent set in $G$ 
contains all $21n(H)$ endvertices of $G$, and 
removing the closed neighborhoods of all these endvertices yields $H'$,
we have 
$$\alpha(G)=\alpha(H')+21n(H)=\alpha(H)+\frac{3}{2}n(H)+21n(H)
=\alpha(H)+\frac{45}{2}n(H),$$
that is, it remains
to show that $\alpha_b(G)=\alpha(G)$.

Let $f:V(G)\to \mathbb{N}_0$ be an optimal independent
broadcast on $G$.  
For every vertex $x$ of $H'$, 
let $L(x)$ be the set of endvertices of $G$
that are at distance $2$ from $x$,
and let $L=\bigcup\limits_{x\in V(H')}L(x)$.
Note that $|L(x)|=3$ if $x\in V(H)$, and 
that $|L(x)|=6$ if $x\in V(H')\setminus V(H)$.

If there is some vertex $x$ in $V(G)\setminus L$ with $f(x)=k$
for some $k\geq 2$,
and $y$ is a vertex in $L$ 
that is closest to $x$,
then changing the value of $f(y)$ to $k$,
and the value of $f(x)$ to $0$
yields an independent broadcast on $G$
of the same weight as $f$.  
Applying this operation iteratively, 
we may assume that 
$$\mbox{$f(x)\geq 2$ only if $x\in L$}.$$
If there is some vertex $y$ in $V(H')$
such that $f(x)\in \{ 2,3\}$ for some vertex $x$ in $L(y)$,
then changing the value of $f$ 
for the at least three vertices in $L(y)$ to $1$
yields an independent broadcast on $G$
whose weight is at least the weight of $f$.
Applying this operation iteratively, 
we may assume that 
$$\mbox{$f(x)\not\in \{ 2,3\}$ for every $x\in V(G)$}.$$
If there is some vertex $y$ in $V(H')\setminus V(H)$
such that $f(x)\in \{ 4,5,6\}$ for some vertex $x$ in $L(y)$,
then changing the value of $f$ 
for the six vertices in $L(y)$ to $1$
yields an independent broadcast on $G$
whose weight is at least the weight of $f$.
Applying this operation iteratively, 
we may assume that 
\begin{eqnarray}\label{e4}
\mbox{$f(x)\not\in \{ 4,5,6\}$ 
for every $y\in V(H')\setminus V(H)$ and every $x\in L(y)$}.
\end{eqnarray}
Let $X=\{ x\in V(G):f(x)>0\}$.
To every vertex $x$ in $X$, 
we assign an independent set $I(x)$ as follows:
\begin{itemize}
\item If $f(x)=1$, then let $I(x)=\{ x\}$.
\item If $f(x)=4$, then, by (\ref{e4}), 
there is some vertex $y$ of $H$
such that $x\in L(y)$. Let $I(x)=\{ y\}\cup L(y)$.
\item If $f(x)=5$, then, by (\ref{e4}),  
there is some vertex $y$ of $H$
such that $x\in L(y)$.
Let $y'$ be some neighbor of $y$ in $H'$,
and let $I(x)=L(y)\cup L(y')$.
\item If $f(x)\geq 6$, 
then there is some vertex $y$ of $H'$ such that $x\in L(y)$.
By (B1), there is a shortest path $P:x_0\ldots x_{\left\lfloor\frac{f(x)}{2}\right\rfloor}$ in $G$ such that $x_0=x$ and $x_2=y$.
Let 
$$I(x)=\bigcup\limits_{i=2}^{\left\lfloor\frac{f(x)}{2}\right\rfloor}L(x_i).$$
\end{itemize}
By construction, $I(x)$ is an independent set for every $x$ in $X$.
Furthermore, if $f(x)\leq 5$ for some $x$ in $X$,
then $|I(x)|\geq f(x)$ is easily verified.
Now, if $f(x)\geq 6$ for some $x$ in $X$, and $P$ is as above, 
then 
$\left\lfloor\frac{f(x)}{2}\right\rfloor\geq 3$,
at least one of the two sets $L(x_2)$ and $L(x_3)$ 
contains six vertices,
and, hence,
$|I(x)|\geq 3\left(\left\lfloor\frac{f(x)}{2}\right\rfloor-1\right)+3
\geq f(x)$.
Altogether,
$I(x)$ is an independent set of order at least $f(x)$ 
for every $x$ in $X$.

Suppose, for a contradiction, 
that there are two distinct vertices $x$ and $x'$ in $X$
such that $I(x)$ and $I(x')$ intersect.
If $f(x)=1$, then, necessarily, $f(x')=4$,
and ${\rm dist}_G(x,x')\leq 2$,
contradicting (B2).
If $f(x),f(x')\in \{ 4,5\}$, then, by (\ref{e4}), 
${\rm dist}_G(x,x')=2$,
and if $f(x)\in \{ 4,5\}$ and $f(x')\geq 6$, then 
${\rm dist}_G(x,x')\leq \left\lfloor\frac{f(x')}{2}\right\rfloor+2<f(x')$,
contradicting (B2).
Finally, if $f(x),f(x')\geq 6$, then 
${\rm dist}_G(x,x')
\leq 
\left\lfloor\frac{f(x)}{2}\right\rfloor
+
\left\lfloor\frac{f(x')}{2}\right\rfloor
\leq \max\{ f(x),f(x')\}$,
again contradicting (B2).
Hence, the sets $I(x)$ for $x$ in $X$ are all disjoint.

Suppose, for a contradiction,
that there are two distinct vertices $x$ and $x'$ in $X$
such that $G$ contains an edge between $I(x)$ and $I(x')$.
Since no two endvertices in $G$ are adjacent,
and $I(x)\subseteq L$ for $x$ in $X$ with $f(x)\geq 5$,
this implies $f(x),f(x')\in \{ 1,4\}$,
which easily implies a contradiction to (B2).
Altogether, it follows that 
$I=\bigcup\limits_{x\in X}I(x)$
is an independent set of order at least $\alpha_b(G)$,
which implies $\alpha(G)\geq \alpha_b(G)$.
Since $\alpha_b(G')\geq \alpha(G')$ 
holds for every graph $G'$, 
this complete the proof.

\section{A polynomial time algorithm for trees}
\label{section:polyAlgoTree}

Throughout this section, let $T$ be a fixed tree of order $n$.

Before we explain the details of our approach, which is based on
dynamic programming, we collect some key observations.  We will
consider certain subtrees $T(u,i)$ of $T$ that contain a vertex $u$
such that all edges of $T$ between $V(T(u,i))$ and $V(T)\setminus
V(T(u,i))$ are incident with $u$,
that is, $V(T(u,i))$ contains $u$ 
and some connected components of $T-u$. 
For every independent
broadcast $f$ on $T$, the restriction of $f$ to $V(T(u,i))$ clearly
satisfies
\begin{quote}
\begin{enumerate}[(C1)]
\item $f(x)\leq {\rm ecc}_T(x)$ for every vertex $x$ of $T(u,i)$.
\item ${\rm dist}_T(x,y)>\max\{ f(x),f(y)\}$ for every two distinct vertices $x$ and $y$ of $T(u,i)$ with $f(x),f(y)>0$.
\end{enumerate}
\end{quote}
Furthermore, if $y$ is a vertex in $V(T)\setminus V(T(u,i))$ with
$f(y)>0$, then $y$ imposes upper bounds on the possible values of $f$
inside $V(T(u,i))$.  More precisely, $f(x)$ must be $0$ for all
vertices $x$ of $T(u,i)$ with ${\rm dist}_T(x,y)\leq f(y)$, and $f(x)$
can be at most ${\rm dist}_T(x,y)-1$ for all vertices $x$ of $T(u,i)$
with ${\rm dist}_T(x,y)>f(y)$. So, in short we have $f(x)=0$ if ${\rm
  dist}_T(x,y)\leq f(y)$ and $f(x)\le {\rm dist}_T(x,y) -1$ otherwise.
Expressed as a function of ${\rm dist}_T(u,x)$ instead of ${\rm
  dist}_T(x,y)$, and using the equality ${\rm dist}_T(u,x)={\rm
  dist}_T(x,y)-{\rm dist}_T(u,y)$, we obtain the condition
$$f(x)\leq g_{(p,q)}({\rm dist}_T(u,x))\mbox{ for every vertex $x$ of $T(u,i)$,}$$
where the function $g_{(p,q)}(d):\mathbb{Z}\to \mathbb{N}_0$
is such that
\begin{eqnarray}
g_{(p,q)}(d)&=&
\begin{cases}
0 & \mbox{, if $d\leq p$, and}\\
d-p+q-1 & \mbox{, if $d\geq p+1$},
\end{cases}\label{e0}\\[3mm]
p&=&f(y)-{\rm dist}_T(u,y)\mbox{, and}\nonumber\\
q&=&f(y).\nonumber
\end{eqnarray}
Note that $q$ is positive, $p$ may be negative, and that $|p|$ and $q$
are both at most the diameter of $T$, which is at most $n$.

One key observation is the following simple lemma.

\begin{lemma}\label{lemma1}
If $(p_1,q_1),\ldots,(p_k,q_k)$ are pairs of integers such that
$-n\leq p_i\leq n$ and $1\leq q_i\leq n$ for every $i$ in $[k]$, then
there is a pair $(p_{\rm in},q_{\rm in})$ of integers such that
$-n\leq p_{\rm in}\leq n$, $1\leq q_{\rm in}\leq n$, and
$$g_{(p_{\rm in},q_{\rm in})}(d)=\min\Big\{
g_{(p_1,q_1)}(d),\ldots,g_{(p_k,q_k)}(d)\Big\}\mbox{ for every
  nonnegative integer $d$}.$$
\end{lemma}
\begin{proof}
The statement follows for $p_{\rm in}=\max\{ p_1,\ldots,p_k\}$ and
$$q_{\rm in}=\min\Big\{ g_{(p_1,q_1)}(p_{\rm
  in}+1),\ldots,g_{(p_k,q_k)}(p_{\rm in}+1)\Big\}.$$ Clearly, we have
$-n\leq p_{\rm in}\leq n$ and $\min\{
g_{(p_1,q_1)}(d),\ldots,g_{(p_k,q_k)}(d)\}=0=g_{(p_{\rm in},q_{\rm
    in})}(d)$ for $d\le p_{\rm in}$.\\ Since $g_{(p_i,q_i)}(p_{\rm
  in}+1)\geq g_{(p_i,q_i)}(p_i+1)=q_i\geq 1$ for every $i$ in $[k]$,
we have $q_{\rm in}\geq 1$.  Furthermore, if $i$ in $[k]$ is such that
$p_{\rm in}=p_i$, then
$$q_{\rm in}\leq g_{(p_i,q_i)}(p_{\rm
  in}+1)=g_{(p_i,q_i)}(p_i+1)=(p_i+1)-p_i+q_i-1=q_i,$$ which implies
$q_{\rm in}\leq q_i\leq n$. 
Finally, notice that for every $p$, $q$, $p'$, and $d$ 
such that $d\ge p+1$ and $p'\ge p$, 
we have
$g_{(p,q)}(d)=d-(p'+1)+g_{(p,q)}(p'+1)$. 
So, 
for every $d\ge p_{\rm in}+1$, 
we have 
\begin{eqnarray*}
\min\{ g_{(p_1,q_1)}(d),\ldots,g_{(p_k,q_k)}(d)\}&=&
d-(p_{\rm in}+1)+\min\{ g_{(p_1,q_1)}(p_{\rm
  in}+1),\ldots,g_{(p_k,q_k)}(p_{\rm in}+1)\}\\
  &=& d-(p_{\rm in}+1)+q_{\rm in}\\
  &=& g_{(p_{\rm in},q_{\rm in})}(d).
\end{eqnarray*}
\end{proof}
Lemma \ref{lemma1} implies that the upper bounds on the possible
values of $f$ inside $V(T(u,i))$ that are imposed by positive values
of $f$ in $V(T)\setminus V(T(u,i))$ can be encoded by just two
integers $p_{\rm in}$ and $q_{\rm in}$ with $-n\leq p_{\rm in}\leq n$
and $1\leq q_{\rm in}\leq n$.  Symmetrically, the upper bounds on the
possible values of $f$ inside $V(T)\setminus V(T(u,i))$ that are
imposed by positive values of $f$ in $V(T(u,i))$, again expressed as a
function of the distance from $u$ in $T$, can be encoded by just two
integers $p_{\rm out}$ and $q_{\rm out}$ with $-n\leq p_{\rm out}\leq
n$ and $1\leq q_{\rm out}\leq n$.

For all $O(n^4)$ possible choices for $((p_{\rm in},q_{\rm
  in}),(p_{\rm out},q_{\rm out}))$, the algorithm determines the
maximum contribution $\sum\limits_{x\in V(T(u,i))}f(x)$ to the weight
of an independent broadcast $f$ on $T$ such that the following conditions hold:
\begin{quote}
\begin{enumerate}
\item[(C3)] $f(x)\leq g_{(p_{\rm in},q_{\rm in})}({\rm dist}_T(u,x))$
for every vertex $x$ of $T(u,i)$.
\item[(C4)] If $f(x)>0$ for some vertex $x$ of $T(u,i)$, then
$$g_{(p_{\rm out},q_{\rm out})}(d)\leq 
g_{(f(x),f(x))}(d+{\rm dist}_T(u,x))\mbox{ for every positive integer $d$.}$$
\end{enumerate}
\end{quote}
Intuitively speaking,
(C3) means that every value of $f$ 
assigned to some vertex of $T(u,i)$
respects the upper bound encoded by $g_{(p_{\rm in},q_{\rm in})}$,
and 
(C4) means that every positive value of $f$ 
assigned to some vertex of $T(u,i)$
imposes an upper bound on the values of $f$ outside
$V(T(u,i))$ that is at least the upper bound
encoded by $g_{(p_{\rm out},q_{\rm out})}$.
The next lemma shows how to check condition (C4) in constant time
for an individual vertex $x$.

\begin{lemma}\label{lemma2}
If $p$, $q$, $f$, and ${\rm dist}$ are integers 
such that $-n\leq p\leq n$ and $q,f,{\rm dist}\in [n]$,
then 
\begin{eqnarray}\label{e1}
g_{(p,q)}(d)\leq 
g_{(f,f)}(d+{\rm dist})\mbox{ for every positive integer $d$}
\end{eqnarray}
if and only if
$\max\Big\{ f-\max\big\{ p,0\big\},q-p\Big\}\leq {\rm dist}$.
\end{lemma}
\begin{proof}
We consider two cases.

First, let $p\leq 0$.
In this case, 
$g_{(p,q)}(d)$ is positive for every positive $d$,
and (\ref{e1}) holds if and only if 
\begin{enumerate}[(i)]
\item $g_{(f,f)}(1+{\rm dist})$ is positive, and
\item $g_{(f,f)}(1+{\rm dist})$ is at least 
$g_{(p,q)}(1)$. 
\end{enumerate}
(i) is equivalent to $1+{\rm dist}\geq f+1$,
and, by (i),
(ii) is equivalent to 
$(1+{\rm dist})-f+f-1\geq 1-p+q-1,$
that is, (i) and (ii) together are equivalent to
$\max\{ f,q-p\}\leq {\rm dist}.$

Next, let $p\geq 1$.
In this case, $g_{(p,q)}(d)$ is positive 
if and only if $d\geq p+1$,
and (\ref{e1}) holds if and only if 
\begin{enumerate}[(i)]
\item $g_{(f,f)}(p+1+{\rm dist})$ is positive, and
\item $g_{(f,f)}(p+1+{\rm dist})$ is at least 
$g_{(p,q)}(p+1)$. 
\end{enumerate}
(i) is equivalent to $p+1+{\rm dist}\geq f+1$,
and, by (i),
(ii) is equivalent to 
$p+1+{\rm dist}-f+f-1
\geq p+1-p+q-1=q,$
that is, (i) and (ii) together are equivalent to
$\max\{ f-p,q-p\}\leq {\rm dist}.$
The statement of the lemma summarizes both cases.
\end{proof}
Now, we explain which subtrees $T(u,i)$ we consider.  We select a
vertex $r$ of $T$, and consider $T$ as rooted in $r$.  For every
vertex $u$ of $T$ that is not a leaf, we fix an arbitrary linear order
on the set of children of $u$.  If $v_1,\ldots,v_k$ are the children
of $u$ in their linear order, then for every $i\in [k]_0$,
let $T(u,i)$ be the subtree of $T$ that contains $u,v_1,\ldots,v_i$ as
well as all descendants of the vertices $v_1,\ldots,v_i$ in $T$.  Note
that $T(u,0)$ contains only $u$, and that $u$ is the only vertex of
$T(u,i)$ that may have neighbors in $T$ outside of $V(T(u,i))$.
Altogether, there are at most 
$(d_T(r)+1) + \sum\limits_{x\in V(T)\setminus \{ r\}}d_T(x)=O(n)$ many choices for $(u,i)$.
For brevity, let $V(u,i)$ denote $V(T(u,i))$.

Let $u$ be a vertex with $k$ children, and let $i\in [k]_0$.  Let
$p_{\rm in}$, $p_{\rm out}$, $q_{\rm in}$, and $q_{\rm out}$ be
integers with $-n\leq p_{\rm in}, p_{\rm out}\leq n$ and 
$1\leq q_{\rm in}, q_{\rm out}\leq n$.  
A function 
$$\mbox{$f:V(u,i)\to \mathbb{N}_0$
is $((p_{\rm in},q_{\rm in}),(p_{\rm out},q_{\rm out}))$-{\it compatible}}$$ 
if conditions (C1), (C2), (C3), and (C4) hold.
Let 
$$\beta((u,i),(p_{\rm in},q_{\rm in}),(p_{\rm out},q_{\rm out}))$$
be the maximum weight $\sum\limits_{x\in V(u,i)}f(x)$ of a function
$f:V(u,i)\to \mathbb{N}_0$ that is 
$((p_{\rm in},q_{\rm in}),(p_{\rm out},q_{\rm out}))$-compatible.

Our next lemma shows that the broadcast independence number 
can be extracted from these values.

\begin{lemma}\label{lemma3}
$\alpha_b(T)=\beta((r,d_T(r)),(-1,n),(n,1))$.
\end{lemma}
\begin{proof}
By definition, $T(r,d_T(r))$ equals $T$.
Since $g_{(-1,n)}(d)\geq n$ for every $d$ in $\mathbb{N}_0$,
condition (C3) is void in view of condition (C1).
Similarly, condition (C4) is void 
by Lemma \ref{lemma2} and condition (C1).
This clearly implies the statement.
\end{proof}
The following lemmas explain how to determine $\beta((u,i),(p_{\rm
  in},q_{\rm in}),(p_{\rm out},q_{\rm out}))$ recursively for all
$O(n^5)$ possible choices for $((u,i),(p_{\rm in},q_{\rm in}),(p_{\rm
  out},q_{\rm out}))$.  The next lemma specifies in particular the
values for leaves.

\begin{lemma}\label{lemma4}
$$\beta((u,0),(p_{\rm in},q_{\rm in}),(p_{\rm out},q_{\rm out}))=
\begin{cases}
0 &\mbox{, if $q_{\rm out}>p_{\rm out}$, and}\\ \min\Big\{ {\rm
  ecc}_T(x),g_{(p_{\rm in},q_{\rm in})}(0),p_{\rm out}\Big\}&\mbox{,
  if $q_{\rm out}\leq p_{\rm out}$.}
\end{cases}
$$
\end{lemma}
\begin{proof}
This follows immediately from (C1), (C3), and (C4) using 
Lemma \ref{lemma2},
$V(u,0)=\{u\}$, 
and ${\rm dist}_T(u,u)=0$.
\end{proof}
The next lemma is the technical key lemma. Recall that $v_1,\dots
,v_k$ denote the children of $u$ in the chosen linear order.
See Figure \ref{fig2}.

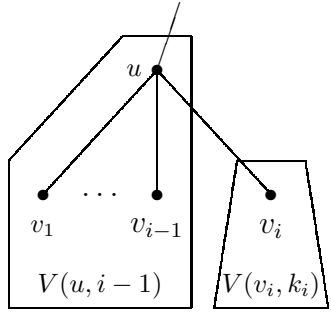
\begin{figure}[H]
\begin{center}
%TeXCAD Picture [2.pic]. Options:
%\grade{\on}
%\emlines{\off}
%\epic{\off}
%\beziermacro{\on}
%\reduce{\on}
%\snapping{\on}
%\pvinsert{% Your \input, \def, etc. here}
%\quality{8.000}
%\graddiff{0.005}
%\snapasp{1}
%\zoom{13.4543}
\unitlength 1.5mm % = 2.845pt
\linethickness{0.4pt}
\ifx\plotpoint\undefined\newsavebox{\plotpoint}\fi % GNUPLOT compatibility
\begin{picture}(28,27)(0,0)
\put(13,21){\circle*{1}}
\put(11,21){\makebox(0,0)[cc]{\footnotesize $u$}}
\put(13,10){\circle*{1}}
\put(3,10){\circle*{1}}
\put(8,10){\makebox(0,0)[cc]{$\ldots$}}
%\emline(3,10)(13,21)
\multiput(3,10)(.0336700337,.037037037){297}{\line(0,1){.037037037}}
%\end
\put(13,21){\line(0,-1){11}}
\put(3,7){\makebox(0,0)[cc]{\footnotesize $v_1$}}
\put(13,7){\makebox(0,0)[cc]{$v_{i-1}$}}
\put(13,21){\line(1,3){2}}
\put(23,10){\circle*{1}}
\put(23,7){\makebox(0,0)[cc]{$v_i$}}
%\emline(23,10)(13,21)
\multiput(23,10)(-.0336700337,.037037037){297}{\line(0,1){.037037037}}
%\end
%\emline(0,13)(10,24)
\multiput(0,13)(.0336700337,.037037037){297}{\line(0,1){.037037037}}
%\end
\put(10,24){\line(1,0){6}}
\put(16,24){\line(0,-1){20}}
\put(16,4){\line(0,-1){4}}
\put(16,0){\line(-1,0){16}}
\put(0,0){\line(0,1){13}}
\put(26,13){\line(-1,0){6}}
%\emline(20,13)(18,0)
\multiput(20,13)(-.03333333,-.21666667){60}{\line(0,-1){.21666667}}
%\end
\put(18,0){\line(1,0){10}}
%\emline(28,0)(26,13)
\multiput(28,0)(-.03333333,.21666667){60}{\line(0,1){.21666667}}
%\end
\put(8,2){\makebox(0,0)[cc]{\footnotesize $V(u,i-1)$}}
\put(23,2){\makebox(0,0)[cc]{\footnotesize $V(v_i,k_i)$}}
\end{picture}
\end{center}
\caption{The situation considered in Lemma \ref{lemma5}.}\label{fig2}
\end{figure}

\begin{lemma}\label{lemma5}
Let $i\in [k]$, and let $v_i$ have $k_i$ children.  Let
$f:V(u,i)\to\mathbb{N}_0$.

If $f$ is $((p_{\rm in},q_{\rm in}),(p_{\rm out},q_{\rm out}))$-compatible,
then there are integers
$$p^{(0)}_{\rm in},p^{(0)}_{\rm out}, 
p^{(1)}_{\rm in}, p^{(1)}_{\rm out}, 
q^{(0)}_{\rm in}, q^{(0)}_{\rm out},
q^{(1)}_{\rm in}, q^{(1)}_{\rm out},$$
with
$-n\leq p^{(0)}_{\rm in}, p^{(0)}_{\rm out},p^{(1)}_{\rm in}, p^{(1)}_{\rm out}\leq n$
and 
$1\leq q^{(0)}_{\rm in}, q^{(0)}_{\rm out},q^{(1)}_{\rm in}, q^{(1)}_{\rm out}\leq n$
such that:
\begin{enumerate}[(i)]
\item The restriction of $f$ to $V(u,i-1)$ is
  $\left(\left(p^{(0)}_{\rm in},q^{(0)}_{\rm
  in}\right),\left(p^{(0)}_{\rm out},q^{(0)}_{\rm
  out}\right)\right)$-compatible.
\item The restriction of $f$ to $V(v_i,k_i)$ is
  $\left(\left(p^{(1)}_{\rm in},q^{(1)}_{\rm
  in}\right),\left(p^{(1)}_{\rm out},q^{(1)}_{\rm
  out}\right)\right)$-compatible.
\item 
$g_{\left(p^{(0)}_{\rm in},q^{(0)}_{\rm in}\right)}(d) =\min\Big\{
  g_{(p_{\rm in},q_{\rm in})}(d),g_{\left(p^{(1)}_{\rm
      out},q^{(1)}_{\rm out}\right)}(d+1)\Big\}$ for every nonnegative
  integer $d$.
\item 
$g_{\left(p^{(1)}_{\rm in},q^{(1)}_{\rm in}\right)}(d) =\min\Big\{
  g_{(p_{\rm in},q_{\rm in})}(d+1),g_{\left(p^{(0)}_{\rm
      out},q^{(0)}_{\rm out}\right)}(d+1)\Big\}$ for every nonnegative
  integer $d$.
\item $g_{(p_{\rm out},q_{\rm out})}(d) \leq \min\Big\{
  g_{\left(p^{(0)}_{\rm out},q^{(0)}_{\rm
      out}\right)}(d),g_{\left(p^{(1)}_{\rm out},q^{(1)}_{\rm
      out}\right)}(d+1)\Big\}$ for every positive integer $d$.
\end{enumerate}
Conversely, if the integers
$$p^{(0)}_{\rm in},p^{(0)}_{\rm out}, 
p^{(1)}_{\rm in}, p^{(1)}_{\rm out}, 
q^{(0)}_{\rm in}, q^{(0)}_{\rm out},
q^{(1)}_{\rm in}, q^{(1)}_{\rm out},$$
with
$-n\leq p^{(0)}_{\rm in}, p^{(0)}_{\rm out},p^{(1)}_{\rm in}, p^{(1)}_{\rm out}\leq n$
and 
$1\leq q^{(0)}_{\rm in}, q^{(0)}_{\rm out},q^{(1)}_{\rm in}, q^{(1)}_{\rm out}\leq n$
are such that conditions (i) to (v) hold, 
then $f$ is $((p_{\rm in},q_{\rm in}),(p_{\rm out},q_{\rm out}))$-compatible.
\end{lemma}
\begin{proof}
First, let $f:V(u,i)\to\mathbb{N}_0$
be $((p_{\rm in},q_{\rm in}),(p_{\rm
  out},q_{\rm out}))$-compatible.

By Lemma \ref{lemma1}, there is a pair $\left(p^{(0)}_{\rm
  out},q^{(0)}_{\rm out}\right)$ such that the function
$g_{\left(p^{(0)}_{\rm out},q^{(0)}_{\rm out}\right)}$ encodes the
upper bounds on the possible values of $f$ outside of $V(u,i-1)$
that are imposed by the positive values of $f$ in $V(u,i-1)$ as a
function of the distance to $u$.  Similarly, there is a pair
$\left(p^{(1)}_{\rm out},q^{(1)}_{\rm out}\right)$ such that the
function $g_{\left(p^{(1)}_{\rm out},q^{(1)}_{\rm out}\right)}$
encodes the upper bounds on the possible values of $f$ outside of
$V(v_i,k_i)$ that are imposed by the positive values of $f$ in
$V(v_i,k_i)$ as a function of the distance to $v_i$.  Since the
distance of a vertex outside of $V(u,i)$ to $v_i$ is one more than
its distance to $u$, condition (C4) implies (v).

By Lemma \ref{lemma1}, the minimum in (iii) equals
$g_{\left(p^{(0)}_{\rm in},q^{(0)}_{\rm in}\right)}$ for a suitable
choice of $\left(p^{(0)}_{\rm in},q^{(0)}_{\rm in}\right)$, 
where we use $g_{(p,q)}(d+1)=g_{(p-1,q)}(d)$ for all integers $p,d$
with $q\ge 1$. 
Similarly, the minimum in (iv) equals
$g_{\left(p^{(1)}_{\rm in},q^{(1)}_{\rm in}\right)}$ for a suitable
choice of $\left(p^{(1)}_{\rm in},q^{(1)}_{\rm in}\right)$,
that is, (iii) and (iv) hold.

Since $f$ is $((p_{\rm in},q_{\rm in}),(p_{\rm out},q_{\rm
  out}))$-compatible, the value of $f$ assigned to a vertex $x$ in
$V(u,i-1)$ is at most $g_{(p_{\rm in},q_{\rm in})}({\rm
  dist}_T(u,x))$.  Furthermore, by (C2) for $f$, the value of $f$
assigned to any vertex $x$ in $V(u,i-1)$ is at most
$g_{\left(p^{(1)}_{\rm out},q^{(1)}_{\rm out}\right)}({\rm
  dist}_T(v_i,x))= g_{\left(p^{(1)}_{\rm out},q^{(1)}_{\rm
    out}\right)}({\rm dist}_T(u,x)+1)$.  Altogether, (i) holds for
$g_{\left(p^{(0)}_{\rm in},q^{(0)}_{\rm in}\right)}$ as in (iii).  By
a completely symmetric argument, it follows that (ii) holds for
$g_{\left(p^{(1)}_{\rm in},q^{(1)}_{\rm in}\right)}$ as in (iv).  This
completes of the proof of the first part of the statement.

Next, let 
$$p^{(0)}_{\rm in}, 
p^{(0)}_{\rm out}, 
p^{(1)}_{\rm in},
p^{(1)}_{\rm out}, 
q^{(0)}_{\rm in}, 
q^{(0)}_{\rm out},
q^{(1)}_{\rm in},q^{(1)}_{\rm out},$$
with
$-n\leq p^{(0)}_{\rm in}, p^{(0)}_{\rm out},p^{(1)}_{\rm in}, p^{(1)}_{\rm out}\leq n$
and 
$1\leq q^{(0)}_{\rm in}, q^{(0)}_{\rm out},q^{(1)}_{\rm in}, q^{(1)}_{\rm out}\leq n$
be such that conditions (i) to (v) hold.

By (i) and (ii), 
$$\mbox{$f$ satisfies (C1)}.$$
By (i) and (ii), 
the restriction of $f$ to $V(u,i-1)$ satisfies (C2), and
the restriction of $f$ to $V(v_i,k_i)$ satisfies (C2).
Let $x\in V(u,i-1)$ and $y\in V(v_i,k_i)$
be such that $f(x),f(y)>0$.
We obtain
\begin{eqnarray*}
f(x) & \stackrel{(i),(C3)}{\leq} & 
g_{\left(p^{(0)}_{\rm in},q^{(0)}_{\rm in}\right)}({\rm dist}_T(u,x))\\
& \stackrel{(iii)}{\leq} & 
g_{\left(p^{(1)}_{\rm out},q^{(1)}_{\rm out}\right)}({\rm dist}_T(u,x)+1)\\
& \stackrel{(ii),(C4)}{\leq} & 
g_{(f(y),f(y))}({\rm dist}_T(u,x)+1+{\rm dist}_T(v_i,y))\\
& = & 
g_{(f(y),f(y))}({\rm dist}_T(x,y)),
\end{eqnarray*}
which implies ${\rm dist}_T(x,y)>f(y)$, because $f(x)$ is positive.
Symmetrically,
we obtain
\begin{eqnarray*}
f(y) & \stackrel{(ii),(C3)}{\leq} & 
g_{\left(p^{(1)}_{\rm in},q^{(1)}_{\rm in}\right)}({\rm dist}_T(v_i,y))\\
& \stackrel{(iv)}{\leq} & 
g_{\left(p^{(0)}_{\rm out},q^{(0)}_{\rm out}\right)}({\rm dist}_T(v_i,y)+1)\\
& \stackrel{(i),(C4)}{\leq} & 
g_{(f(x),f(x))}({\rm dist}_T(v_i,y)+1+{\rm dist}_T(u,x))\\
& = & 
g_{(f(x),f(x))}({\rm dist}_T(x,y)),
\end{eqnarray*}
which implies ${\rm dist}_T(x,y)>f(x)$, because $f(y)$ is positive.
Altogether, it follows that 
$$\mbox{$f$ satisfies (C2)}.$$
Since, by (iii), 
$g_{\left(p^{(0)}_{\rm in},q^{(0)}_{\rm in}\right)}(d)\leq g_{\left(p_{\rm in},q_{\rm in}\right)}(d)$,
and, by (iv), 
$g_{\left(p^{(1)}_{\rm in},q^{(1)}_{\rm in}\right)}(d)\leq g_{\left(p_{\rm in},q_{\rm in}\right)}(d+1)$
for every nonnegative integer $d$, we have 
$$\mbox{$f$ satisfies (C3)}.$$
If $x\in V(u,i-1)$ is such that $f(x)>0$, then,
\begin{eqnarray*}
g_{(p_{\rm out},q_{\rm out})}(d)
&\stackrel{(v)}{\leq} &
g_{\left(p^{(0)}_{\rm out},q^{(0)}_{\rm out}\right)}(d)\\
&\stackrel{(i),(C4)}{\leq} &
g_{(f(x),f(x))}(d+{\rm dist}_T(u,x))
\mbox{ for every positive integer $d$.}
\end{eqnarray*}
Similarly, 
if $x\in V(v_i,k_i)$ is such that $f(x)>0$, then,
\begin{eqnarray*}
g_{(p_{\rm out},q_{\rm out})}(d)
&\stackrel{(v)}{\leq} &
g_{\left(p^{(1)}_{\rm out},q^{(1)}_{\rm out}\right)}(d+1)\\
&\stackrel{(ii),(C4)}{\leq} &
g_{(f(x),f(x))}(d+1+{\rm dist}_T(v_i,x))\\
&=&
g_{(f(x),f(x))}(d+{\rm dist}_T(u,x))
\mbox{ for every positive integer $d$.}
\end{eqnarray*}
Altogether,
$$\mbox{$f$ satisfies (C4)},$$
which completes the proof.
\end{proof}

The next lemma is a consequence of Lemma \ref{lemma5}.

\begin{lemma}\label{lemma6}
Let $i\in [k]$, and let $v_i$ have $k_i$ children.

$\beta((u,i),(p_{\rm in},q_{\rm in}),(p_{\rm out},q_{\rm out}))$
equals the maximum of
$$\beta\left((u,i-1),\left(p^{(0)}_{\rm in},q^{(0)}_{\rm
  in}\right),\left(p^{(0)}_{\rm out},q^{(0)}_{\rm out}\right)\right)
+\beta\left((v_i,k_i),\left(p^{(1)}_{\rm in},q^{(1)}_{\rm
  in}\right),\left(p^{(1)}_{\rm out},q^{(1)}_{\rm
  out}\right)\right),$$ where the maximum extends over all choices of
$p^{(0)}_{\rm in}$, $p^{(0)}_{\rm out}$, $p^{(1)}_{\rm in}$,
$p^{(1)}_{\rm out}$, $q^{(0)}_{\rm in}$, $q^{(0)}_{\rm out}$,
$q^{(1)}_{\rm in}$, and $q^{(1)}_{\rm out}$, with $-n\leq p^{(0)}_{\rm
  in}, p^{(0)}_{\rm out},p^{(1)}_{\rm in}, p^{(1)}_{\rm out}\leq n$
and $1\leq q^{(0)}_{\rm in}, q^{(0)}_{\rm out},q^{(1)}_{\rm in},
q^{(1)}_{\rm out}\leq n$ that satisfy the conditions (iii), (iv), and
(v) from Lemma \ref{lemma5}.
\end{lemma}
\begin{proof}
This follows from Lemma~\ref{lemma5} and the fact that $V(u,i)$ is
the disjoint union of $V(u,i-1)$ and $V(v_i,k_i)$.
Notice that when considering the values
$\beta\left((u,i-1),\left(p^{(0)}_{\rm in},q^{(0)}_{\rm in}\right),
\left(p^{(0)}_{\rm out},q^{(0)}_{\rm out}\right)\right)$ and 
$\beta\left((v_i,k_i),\left(p^{(1)}_{\rm in},
q^{(1)}_{\rm in}\right),
\left(p^{(1)}_{\rm out},q^{(1)}_{\rm out}\right)\right)$, 
we know, by definition, 
that the corresponding restricted independent broadcasts on $T(u,i-1)$ and $T(v_i,k_i)$ are 
$\left(\left(p^{(0)}_{\rm in},q^{(0)}_{\rm in}\right),
\left(p^{(0)}_{\rm out},q^{(0)}_{\rm out}\right)\right)$-compatible and 
$\left(\left(p^{(1)}_{\rm in},q^{(1)}_{\rm in}\right),
\left(p^{(1)}_{\rm out},q^{(1)}_{\rm out}\right)\right)$-compatible,
respectively.
As the choice of the parameters satisfies the conditions (iii), (iv), and (v) 
from Lemma \ref{lemma5}, 
we have seen in the proof of that lemma 
that the common extension of the two restricted independent broadcasts 
is a restricted independent broadcast on $T(u,i)$.
\end{proof}
We are now ready for the following.

\begin{proof}[Proof of Theorem \ref{theorem2}]
By Lemmas \ref{lemma4} and \ref{lemma6}, 
each of the $O(n^5)$ values
$$\beta((u,i),(p_{\rm in},q_{\rm in}),(p_{\rm out},q_{\rm out}))$$
can be determined in time $O(n^4)$, by 
\begin{itemize}
\item processing the vertices $u$ of $T$ 
in an order of nonincreasing depth within $T$,
\item considering the $O(n)$ possible values for each of the four integers $p^{(0)}_{\rm out}$, 
$p^{(1)}_{\rm out}$, 
$q^{(0)}_{\rm out}$, and
$q^{(1)}_{\rm out}$,
\item checking condition (v) from Lemma \ref{lemma5} in constant time,
\item determining the four integers
$p^{(0)}_{\rm in}$, 
$p^{(1)}_{\rm in}$, 
$q^{(0)}_{\rm in}$, and
$q^{(1)}_{\rm in}$
as in (iii) and (iv) from Lemma \ref{lemma5} in constant time,
and 
\item adding 
$\beta\left((u,i-1),\left(p^{(0)}_{\rm in},q^{(0)}_{\rm in}\right),\left(p^{(0)}_{\rm out},q^{(0)}_{\rm out}\right)\right)$
and 
$\beta\left((v_i,k_i),\left(p^{(1)}_{\rm in},q^{(1)}_{\rm in}\right),\left(p^{(1)}_{\rm out},q^{(1)}_{\rm out}\right)\right)$.
\end{itemize}
Now, Lemma \ref{lemma3} implies the statement.
\end{proof}
It may be possible --- yet tedious --- to extend Theorem \ref{theorem2} 
to graphs of bounded treewidth.


\begin{thebibliography}{}
\bibitem{ahboso} M. Ahmane, I. Bouchemakh, E. Sopena, On the broadcast independence number of caterpillars, Discrete Applied Mathematics 244 (2018) 20-35.
\bibitem{bera} S. Bessy, D. Rautenbach,
Relating broadcast independence and independence,
manuscript 2018.
\bibitem{blhehoma} J.R.S. Blair, P. Heggernes, S. Horton, F. Manne, Broadcast domination algorithms for interval graphs, series-parallel graphs and trees, Congressus Numerantium 169 (2004) 55-77.
\bibitem{boze} I. Bouchemakh, M. Zemir, On the broadcast independence number of grid graph, Graphs and Combinatorics 30 (2014) 83-100.
\bibitem{dadehe} J. Dabney, B.C. Dean, S.T. Hedetniemi, A linear-time algorithm for broadcast domination in a tree, Networks 53 (2) (2009) 160-169.
\bibitem{duerhahehe} J.E. Dunbar, D.J. Erwin, T.W. Haynes, S.M. Hedetniemi, S.T. Hedetniemi, Broadcasts in graphs, Discrete Applied Mathematics 154 (2006) 59-75.
\bibitem{er} D.J. Erwin, Cost domination in graphs, (Ph.D. thesis), Western Michigan University, 2001.
\bibitem{gajo} M. Garey and D. Johnson, The rectilinear Steiner tree problem is NP-complete, SIAM Journal on Applied Mathematics 32 (1977) 826-834.
\bibitem{he} S.T. Hedetniemi, Unsolved algorithmic problems on trees, AKCE International Journal of Graphs and Combinatorics 3 (1) (2006) 1-37.
\bibitem{helo} P. Heggernes, D. Lokshtanov, Optimal broadcast domination in polynomial time, Discrete Mathematics 36 (2006) 3267-3280.
\bibitem{zu} D. Zuckerman, Linear degree extractors and the inapproximability of max clique and chromatic number, Theory of Computing (2007) 103-128.
\end{thebibliography}
\end{document}